\documentclass[11pt]{article}
\usepackage[a4paper]{geometry}
\usepackage{amsthm}
\usepackage{amsmath}
\usepackage{amssymb}
\usepackage{amsfonts}
\usepackage{epsfig,graphicx}
\usepackage{subfigure}
\usepackage{wrapfig}

\begin{document}

\newtheorem{Theorem}{Theorem}[section]
\newtheorem{Proposition}[Theorem]{Proposition}
\newtheorem{Corollary}[Theorem]{Corollary}
\newtheorem{Lemma}[Theorem]{Lemma}
\newtheorem{Condition}[Theorem]{Condition}
\newtheorem{Integrator}[Theorem]{Integrator}

\theoremstyle{definition}
\newtheorem{Definition}{Definition}[section]
\newtheorem{Remark}{Remark}[section]
\newtheorem{Example}{Example}[section]

\newcommand{\todo}[1]{\vspace{5 mm}\par \noindent
\marginpar{\textsc{ToDo}}
\framebox{\begin{minipage}[c]{0.95 \textwidth}
\tt #1 \end{minipage}}\vspace{5 mm}\par}

\title{Explicit symplectic approximation of nonseparable Hamiltonians: algorithm and long time performance}

\author{Molei Tao}
\date{}

\maketitle

\begin{abstract}
Explicit symplectic integrators have been important tools for accurate and efficient approximations of mechanical systems with separable Hamiltonians. For the first time, the article proposes for arbitrary Hamiltonians similar integrators, which are explicit, of any even order, symplectic in an extended phase space, and with pleasant long time properties. They are based on a mechanical restraint that binds two copies of phase space together. Using backward error analysis, KAM theory, and additional multiscale analysis, an error bound of $\mathcal{O}(T\delta^l \omega)$ is established for integrable systems, where $T$, $\delta$, $l$ and $\omega$ are respectively the (long) simulation time, step size, integrator order, and some binding constant. For non-integrable systems with positive Lyapunov exponents, such an error bound is generally impossible, but satisfactory statistical behaviors were observed in a numerical experiment with a nonlinear Schr\"{o}dinger equation.
\end{abstract}

\section{Introduction and the algorithm}
Symplectic integrators preserve the phase space volume $dq \wedge dp$ like the exact Hamiltonian flow, and thus are the preferred approach for long-time simulations of mechanical systems \cite{sanz1992symplectic,MaWe:01,Hairer06,calvo1995accurate, quispel1998volume}.
Explicit symplectic integration has been extensively studied for separable Hamiltonian (i.e. $H(q,p)=K(p)+V(q)$; see \cite{ruth1983canonical, feng1986difference, yoshida1990construction, wisdom1991symplectic, sanz1992symplectic, MaWe:01, mclachlan2002splitting, leimkuhler2004simulating, Hairer06, quispel1998volume}), but much less so for nonseparable systems. However, nonseparable Hamiltonians model important problems, such as a finite-dimensional representation of nonlinear Schr\"{o}dinger equation \cite{colliander2010transfer},
nearly integrable systems in action-angle coordinates (see \cite{henrard1983second,murray1997diffusive,li2014eccentricity,li2014chaos} for astrophysical examples), charged particle dynamics \cite{forest2006geometric,qin2008variational,Ta16}, mechanical systems in a rotating frame \cite{koon2000heteroclinic}, molecular dynamics with thermostats \cite{bond1999nose,sturgeon2000symplectic}, time regularized mechanical systems \cite{worthington2012study}, classical systems with post-Newtonian correction that approximates general relativity effects \cite{landau2013classical, buonanno2006transition}, rigid body dynamics \cite{kol1997symplectic}, pendula dynamics \cite{tao2016variational}, and the scattering of electromagnetic waves by attenuating materials \cite{huang2006symplectic}.

A common misunderstanding is that a symplectic integration has to be implicit (and hence computationally expensive) when the Hamiltonian is nonseparable. In fact, explicit symplectic integrations have been made possible for several subclasses of nonseparable Hamiltonians \cite{calvo1994numerical,sturgeon2000symplectic, blanes2002symplectic,mclachlan2004explicit, chin2009explicit, wu2003explicit,forest2006geometric,Ta16}. Notably, a first step toward explicit approximation of arbitrary $H$ was recently made in \cite{pihajoki2015explicit} by considering a larger system in an extended phase space; although the method integrator proposed there is only accurate for a short time, the idea of extended phase space is a building block of this research.

We approximate the flow of an arbitrary nonseparable $H(Q,P)$. For the first time, a generic, high-order, explicit and symplectic integrator is proposed with a provable pleasant long time performance. This is achieved by considering an augmented Hamiltonian
\[
	\bar{H}(q,p,x,y):=H_A+H_B+\omega H_C
\]
in an extended phase space with symplectic 2-form $dq \wedge dp+dx \wedge dy$, where $H_A:=H(q,y)$ and $H_B:=H(x,p)$ correspond to two copies of the original system with mixed-up positions and momenta, $H_C:=\|q-x\|_2^2/2+\|p-y\|_2^2/2$ is an artificial restraint, and $\omega$ is a constant that controls the binding of the two copies.

First of all, note the initial value problems
\[	\begin{cases}
		\dot{Q}=\partial_P H(Q,P),	& Q(0)=Q_0 \\
		\dot{P}=-\partial_Q H(Q,P),	& P(0)=P_0
	\end{cases} \quad \text{and} \quad
	\begin{cases}
		\dot{q}=\partial_p \bar{H}(q,p,x,y),	& q(0)=Q_0 \\
		\dot{p}=-\partial_q \bar{H}(q,p,x,y),	& p(0)=P_0 \\
		\dot{x}=\partial_y \bar{H}(q,p,x,y),	& x(0)=Q_0 \\
		\dot{y}=-\partial_x \bar{H}(q,p,x,y),	& y(0)=P_0
	\end{cases}
\]
have the same exact solution in the sense that $q(t)=x(t)=Q(t), p(t)=y(t)=P(t)$. This is because the form of $\bar{H}$ turns the second system into
\[	\begin{cases}
	\dot{q} &= \partial_p H(x,p)+\omega(p-y) \\
	\dot{p} &= -\partial_q H(q,y)-\omega(q-x) \\
	\dot{x} &= \partial_y H(q,y)+\omega(y-p) \\
	\dot{y} &= -\partial_x H(x,p)-\omega(x-q)
\end{cases}, \]
which upon the substitution $q(t)=x(t)=Q(t),p(t)=y(t)=P(t)$ becomes the first system. Uniqueness of ODE solution rules out the possibility of disagreed other solutions.

Second, it is possible to construct high-order symplectic integrators for $\bar{H}$ with explicit updates.
Denote respectively by $\phi_{H_A}^\delta, \phi_{H_B}^\delta, \phi_{\omega H_C}^\delta$ the time-$\delta$ flow of $H_A$, $H_B$, $\omega H_C$. Exact expressions of these flows can be explicitly obtained as:
\begin{align}
	&\phi_{H_A}^\delta: \begin{bmatrix} q \\ p \\ x \\ y \end{bmatrix} 
						\mapsto
						\begin{bmatrix}	q \\ p-\delta \partial_q H(q,y) \\ x+\delta \partial_y H(q,y) \\ y \end{bmatrix},	\quad
	 \phi_{H_B}^\delta: \begin{bmatrix} q \\ p \\ x \\ y \end{bmatrix}
						\mapsto
						\begin{bmatrix} q+\delta \partial_p H(x,p) \\ p \\ x \\ y-\delta \partial_x H(x,p) \end{bmatrix},	\quad	\nonumber\\
	&\phi_{\omega H_C}^\delta:	\begin{bmatrix} q \\ p \\ x \\ y \end{bmatrix}
						\mapsto
						\frac{1}{2} \begin{bmatrix} \begin{pmatrix} q+x \\ p+y \end{pmatrix} + R(\delta) \begin{pmatrix} q-x \\ p-y \end{pmatrix} \\ \begin{pmatrix} q+x \\ p+y \end{pmatrix} - R(\delta) \begin{pmatrix} q-x \\ p-y \end{pmatrix} \end{bmatrix}, \quad
	\text{where } R(\delta):=\begin{bmatrix} \cos(2 \omega \delta)I & \sin(2 \omega \delta)I \\ -\sin(2 \omega \delta)I & \cos(2 \omega \delta)I \end{bmatrix}.
	\label{eqn_flowMaps}
\end{align}
Then we construct a numerical integrator that approximates $\bar{H}$ by composing these maps: it is well known that
\begin{equation}
	\phi_2^\delta := \phi_{H_A}^{\delta/2} \circ \phi_{H_B}^{\delta/2} \circ \phi_{\omega H_C}^\delta \circ \phi_{H_B}^{\delta/2} \circ \phi_{H_A}^{\delta/2},
	\label{eqn_2ndOrderIntegrator}
\end{equation}
commonly named as Strang splitting, has a 3rd-order local error (thus a 2nd-order method), and is a symmetric method. Arbitrary high (even) order integrator can also be obtained. The $l$-th order version will have an update map $\phi_l^\delta$ given by, for instance, `the triple jump' \cite{forest1989canonical,suzuki1990fractal, yoshida1990construction, mclachlan2002splitting}:
\begin{equation}
	\phi_l^\delta := \phi_{l-2}^{\gamma_l \delta} \circ \phi_{l-2}^{(1-2\gamma_l) \delta} \circ \phi_{l-2}^{\gamma_l \delta}, \quad \text{where }\gamma_l=\frac{1}{2-2^{1/(l+1)}}.
	\label{eqn_highOrderIntegrator}
\end{equation}
Each update $\phi_l^\delta$ constitutes one of our proposed integrators, which are symplectic because each flow is symplectic. Each one produces a discrete trajectory
\[
	\begin{bmatrix} q_N \\ p_N \\ x_N \\ y_N \end{bmatrix} := \left(\phi_l^\delta\right)^N \begin{bmatrix} Q(0) \\ P(0) \\ Q(0) \\ P(0) \end{bmatrix},
\]
where $q_N, p_N$ (and $x_N, y_N$ too) approximate the exact solution $Q(N\delta), P(N\delta)$. 

This approximation needs justification, because although the exact solutions of $H$ and $\bar{H}$ agree, truncation errors in the numerical solution of $\bar{H}$ may lead to large global error after $\mathcal{O}(1)$ time. 
We'll see this won't be the case if the system is integrable. To discuss the idea, note \cite{pihajoki2015explicit} 
considered the Hamiltonian $H_A+H_B$ without the binding, i.e. $\bar{H}$ with $\omega=0$. The resulting integrator did produce $(q,p)$ and $(x,y)$ that well approximate $Q,P$ till $\mathcal{O}(1)$ time, but then they quickly diverge. A fix was suggested in \cite{pihajoki2015explicit} based on an extra phase space mixing substep for inducing a coupling between $(q,p)$ and $(x,y)$; unfortunately, symplecticity is lost due to this substep. We replace this substep using $\omega H_C$. This is because, under reasonable assumptions, the near conservation of $\bar{H}$ by its symplectic integration (established by backward error analysis; see \cite{Hairer06}) will imply the boundedness of $\omega H_C$, and thus that $q-x,p-y$ are at most $\mathcal{O}(1/\sqrt{\omega})$, which prevents the undesired divergence.

One may worry whether large $\omega$ requires small $\delta$, which would undermine the computational efficiency gained by an explicit integrator. Also, will a finite $\omega$ introduce another source of error besides truncation error? In addition, one is interested in whether $q-Q, p-P$ are small, but not $q-x$ and $p-y$. Section \ref{sec_integrableCase} will, for integrable $\bar{H}$, bound $\|q-Q\|$ and $\|p-P\|$ till large time. 
Section \ref{sec_nonIntegrableCase} will show if $H$ is not integrable, although long time accuracy of trajectory should not be expected,
it is still possible to numerically capture statistical behaviors of the system, at least in a nonlinear Schr\"{o}dinger equation example. The difference between integrabilities of $H$ and $\bar{H}$ is discussed in section \ref{sec_integrabilityExtendedSys}, which further explains why $\omega$ should be larger than a threshold.

\section{Integrable problems: linear growth of long time approximation error}
\label{sec_integrableCase}
Provided that $\bar{H}$ corresponds to an integrable system (which will be the case if, roughly speaking, $H$ is integrable and $\omega\geq \omega_0$ for some constant $\omega_0$; see section \ref{sec_integrabilityExtendedSys}), we will demonstrate that the proposed $l^{th}$-order integrator (eqn. \ref{eqn_highOrderIntegrator}) has a numerical error of
\[
	\mathcal{O}(T \delta^l \omega)
\]
till at least $T=\mathcal{O}(\min(\delta^{-l} \omega^{-1}, \omega^{1/2}))$. Numerical results consistent with this bound for even larger $T$ values will then be shown.

For long time simulation, this linear growth with $T$ is advantageous to non-symplectic integrators, whose errors can grow exponentially, e.g., $\mathcal{O}(e^{CT} \delta^l \omega^l)$ \cite{MR1227985}. It also improves the pioneering symplectic integrator in \cite{pihajoki2015explicit}, which becomes inaccurate after $T=\mathcal{O}(1)$.
One can see $\omega \gtrsim \omega_0$ is sufficient for accuracy, 
and the extra error introduced by a finite $\omega$ vanishes as $\delta\rightarrow 0$. Interestingly, an $\omega$ too large is actually discouraged by this error bound. In addition, $\delta \ll \omega^{-1/l}$ is sufficient, and although there is still a trade-off between accuracy and efficiency, a larger $l$ allows $\delta$ to be much larger than $o(1/\omega)$, i.e. no need to resolve the oscillation induced by $\omega H_C$. Accuracy and efficiency are thus simultaneously improved.

The main idea for establishing this bound is to view the numerical solution as discrete samples of the exact solution of some near-by Hamiltonian $\tilde{H}$ (i.e. backward error analysis), and characterize the distance between $\tilde{H}$ and $\bar{H}$ as a small parameter $\epsilon$. An application of KAM theory  \cite{kolmogorov1954conservation,arnol1963proof,moser1962invariant,poschel1982integrability,eliasson1996absolutely,
giorgilli1997kolmogorov} then bounds the differences between action and angle variables in $\tilde{H}$ and $\bar{H}$. Such bounds will specify how these two systems deviate in $q,p$ coordinates and hence quantify the numerical error, because the flows of $\bar{H}$ and $\tilde{H}$ respectively correspond to the exact solution of $H$ and the numerical solution. Similar techniques have been established (see Chap X of \cite{Hairer06} for a review), and the main novelty of our derivation is a refined estimation of $\epsilon$ combined with these techniques. Specifically,
\begin{enumerate}
\item
\label{item_step2}
Denote by $t:=\omega T$ and $h:=\omega \delta$ new time variable and step; in the new time, the Hamiltonian $\bar{H}$ becomes $\frac{1}{\omega}H_A + \frac{1}{\omega}H_B+H_C$. Since Lie bracket of Hamiltonian vector fields corresponds to Poisson bracket of Hamiltonians \cite{MaRa10}, repeated applications of Baker-Campbell-Hausdorff formula show that the time-rescaled version of \eqref{eqn_2ndOrderIntegrator},
\[
	\phi_2^h := \phi_{H_A/\omega}^{h/2} \circ \phi_{H_B/\omega}^{h/2} \circ \phi_{H_C}^h \circ \phi_{H_B/\omega}^{h/2} \circ \phi_{H_A/\omega}^{h/2},
\]
corresponds to the symplectic time-$h$ flow of the Hamiltonian $\tilde{H}= \bar{H}+R$, where the perturbative remainder $R$ is defined by
{
\small
\begin{align}
	R	:=  & - \frac{1}{24} h^2 \left\{\frac{1}{\omega}H_A,\left\{\frac{1}{\omega}H_A,\frac{1}{\omega}H_B\right\}\right\} 
		 		- \frac{1}{24} h^2 \left\{\frac{1}{\omega}H_A,\left\{\frac{1}{\omega}H_A,H_C\right\}\right\} \nonumber\\
		& + \frac{1}{12} h^2 \left\{\frac{1}{\omega}H_B,\left\{\frac{1}{\omega}H_B,\frac{1}{\omega}H_A\right\}\right\} 
		 		- \frac{1}{24} h^2 \left\{\frac{1}{\omega}H_B,\left\{\frac{1}{\omega}H_B,H_C\right\}\right\} \nonumber\\
		& + \frac{1}{12} h^2 \left\{H_C,\left\{H_C,\frac{1}{\omega}H_A\right\}\right\} 
		 		+ \frac{1}{12} h^2 \left\{H_C,\left\{H_C,\frac{1}{\omega}H_B\right\}\right\} \nonumber\\
		& + \frac{1}{12} h^2 \left\{\frac{1}{\omega}H_B,\left\{H_C,\frac{1}{\omega}H_A\right\}\right\} 
		 		+ \frac{1}{12} h^2 \left\{H_C,\left\{\frac{1}{\omega}H_B,\frac{1}{\omega}H_A\right\}\right\} + \mathcal{O}(h^4).
	\label{eqn_backwardH}
\end{align}
}Higher-order methods can be similarly analyzed: it is known that a $l$-th order integrator based on Hamiltonian splitting samples the exact flow of an $\mathcal{O}(h^l)$ perturbation of the exact Hamiltonian, where the perturbation is a sum of terms expressible using at least $l$ nested Poisson brackets of $H_A/\omega$, $H_B/\omega$, and $H_C$ \cite{Hairer06}. Therefore, the magnitude of $R$ is at most $\mathcal{O}(h^l/\omega)$, because any nonzero nested Poisson bracket has to involve $H_A/\omega$ or $H_B/\omega$ at least once (otherwise, $\{H_C,H_C\}=0$ leads to a zero result).

\item
Assuming $q,p,x,y$ and $R$ are bounded, conservation of $\tilde{H}$ implies boundedness of $H_C$, which leads to $\|q-x\|=\mathcal{O}(1/\sqrt{\omega})$ and $\|p-y\|=\mathcal{O}(1/\sqrt{\omega})$. Under a canonical transformation $\alpha=q-x$, $\beta=p-y$, $\bar{Q}=(q+x)/2$, $\bar{P}=(p+y)/2$,
\[
	\tilde{H}=\frac{1}{\omega} H\left(\bar{Q}+\frac{\alpha}{2},\bar{P}-\frac{\beta}{2}\right) + \frac{1}{\omega} H\left(\bar{Q}-\frac{\alpha}{2},\bar{P}+\frac{\beta}{2}\right) + \frac{1}{2} \|\alpha\|^2 + \frac{1}{2} \|\beta\|^2 + R
\]
for some small $R(\alpha,\beta,\bar{Q},\bar{P})$. Assuming $H(\cdot,\cdot)$ is analytic, then $\tilde{H}$ is analytic in $\bar{Q},\bar{P},\alpha,\beta$. Since $\bar{Q},\bar{P}$ remain bounded, $\alpha(0)=\beta(0)=\textbf{0}$, and $\alpha(t),\beta(t)=\mathcal{O}(1/\sqrt{\omega})$, a multiscale analysis based on normal form shows that $\alpha(t),\beta(t)=\mathcal{O}(1/\omega)$ till at least $t=\mathcal{O}(\omega^{3/2})$ (see appendix \ref{sec_appRefinedEstimate}); this corresponds to $T=\mathcal{O}(\sqrt{\omega})$ in the original time, which is still a long time.

\item
We now refine the perturbation magnitude estimation. Each time $H_A/\omega$ or $H_B/\omega$ appears in a nested Poisson bracket term in \eqref{eqn_backwardH}, that term gets scaled by $1/\omega$. On the other hand, when $H_C$ appears, by the definition of Poisson bracket,
\[
	\{X,H_C\}=\frac{\partial X}{\partial q} \cdot \frac{\partial H_C}{\partial p}
	       -\frac{\partial X}{\partial p} \cdot \frac{\partial H_C}{\partial q}
	       +\frac{\partial X}{\partial x} \cdot \frac{\partial H_C}{\partial y}
	       -\frac{\partial X}{\partial y} \cdot \frac{\partial H_C}{\partial x}.
\]
Since derivatives of $H_C$ are $\pm (q-x)$ or $\pm (p-y)$, all of them lead to a scaling by $\mathcal{O}(1/\omega)$ too. Therefore, any $l$-nested Poisson bracket term containing $l+1$ Hamiltonians is actually $\mathcal{O}(h^l/\omega^{1+l})$ till at least $t=\mathcal{O}(\omega^{3/2})$.

\item
Let $\epsilon=h^l/\omega^l$. Then $\tilde{H}-\bar{H}=R=\mathcal{O}(\epsilon/\omega)$ till at least $t=\mathcal{O}(\omega^{3/2})$, and it is a $\mathcal{O}(\epsilon)$ perturbation when compared with $\bar{H}$. A KAM type estimate (see \cite{Hairer06}, Chap X) shows that under technical conditions (e.g., Diophantine and non-degenerate initial condition) the solution of
$\tilde{H}$ differs from that of $\bar{H}$ by $\mathcal{O}(t \epsilon)$ for at least $t=\mathcal{O}(\min(\epsilon^{-1}, \omega^{3/2}))$ (minimum of the two because $R$ estimate is only valid till $t=\mathcal{O}(\omega^{3/2})$).

Since $\tilde{H}$ corresponds to the numerical solution, and $\bar{H}$ corresponds to the exact solution, the numerical error is thus $\mathcal{O}(t h^l / \omega ^l)$. Converting back to the original time, this corresponds to a numerical error of
$\mathcal{O}(T \delta^l \omega)$
till at least $T=\mathcal{O}(\min(\delta^{-l} \omega^{-1}, \omega^{1/2}))$,
where $\delta$ is the step size used by the proposed integrator \eqref{eqn_highOrderIntegrator} in the original time.
\end{enumerate}

\subsection{A numerical demonstration}
\label{sec_1DOF}
We now demonstrate the $\mathcal{O}(T \delta^l \omega)$ error bound on a 1 degree of freedom system with $H(Q,P)=(Q^2+1)(P^2+1)/2$. This is a system with obtainable exact solution, and thus long-time numerical errors can be accurately quantified.


To derive the exact solution, note energy level sets in this system are closed curves
symmetric about $P=0$ and about $Q=0$. 
Half period of the dynamics is governed by
\[
	\dot{Q}=(1+Q^2)\sqrt{\frac{2E}{1+Q^2}-1}, \text{ where } P(0)=0, ~ E=(1+Q(0)^2)/2, \text{ and } Q(0)<0 \text{ is assumed},
\]
and this dynamics is till $\mathcal{T}>0$ such that $Q(\mathcal{T})=-Q(0)$. Its exact solution can be obtained using Jacobi's elliptic function, namely
\[
	Q(t)=Q(0) \text{ cn} \left( \left. t\sqrt{1+Q(0)^2} \right| \frac{Q(0)^2}{1+Q(0)^2} \right),
\]
and thus $\mathcal{T}$ can be obtained using Gauss hypergeometric function
\[
	\mathcal{T}=\pi \,_2F_1 (0.5,0.5;1, -Q(0)^2);
\]
see \cite{abramowitz1964handbook} for more details about these special functions. The exact solution of $H$ is now available till arbitrary time, because it is $2\mathcal{T}$-periodic, and its other half period (for time $[(2n+1)\mathcal{T},2n \mathcal{T}],n \in \mathbb{Z}$) is symmetric to the previously obtained half (for time $[2n \mathcal{T},(2n+1) \mathcal{T}],n \in \mathbb{Z}$) about the origin.

\begin{figure}[h]
\centering
\hspace{-10pt}
\subfigure[Proposed method $\phi_4^\delta$]{
\includegraphics[width=0.5\textwidth]{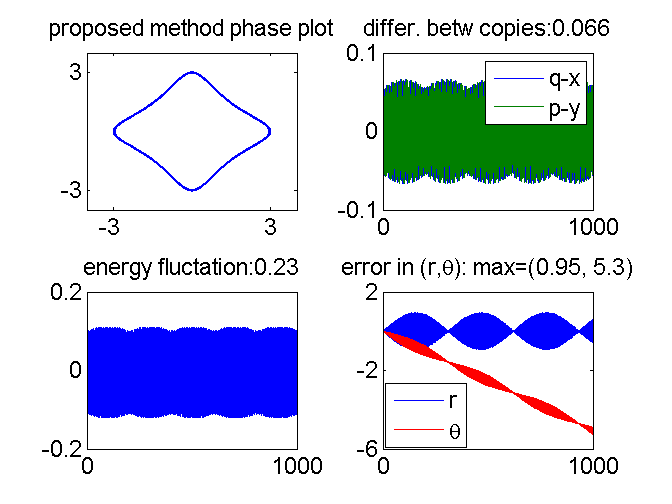}
\label{fig_1DOF_comparison_new}
}
\hspace{-15pt}
\subfigure[4th-order Runge-Kutta]{
\includegraphics[width=0.5\textwidth]{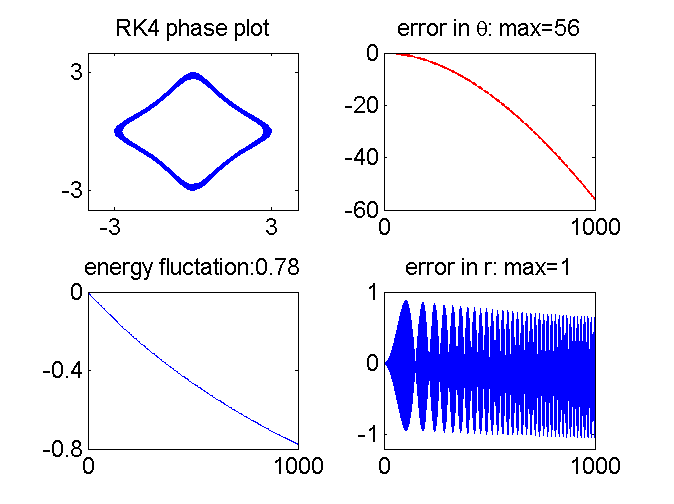}
\label{fig_1DOF_comparison_RK4}
}
\caption{\footnotesize Comparison between the proposed integrator and a classical nonsympletic method. Timestep $\delta=0.1$ for both methods, $Q(0)=-3,P(0)=0$, $\omega=20$, and $T=1000$.}
\label{fig_1DOF_comparison}
\end{figure}

Figure \ref{fig_1DOF_comparison} compares long time simulations by a 4th-order proposed method and the 4th-order Runge-Kutta. Errors of $Q,P$ trajectories are investigated in polar coordinates (figure \ref{fig_1DOF_comparison_RK4} right bottom panel and figure \ref{fig_1DOF_comparison_RK4} right column). Runge-Kutta has an exponentially increasing phase error and a drifting amplitude error corresponding to undesired numerical viscosity, whereas the proposed method has linearly growing phase error and bounded amplitude error, consistent with the error bound but extending much beyond $\mathcal{O}(\sqrt{\omega})$ time.

One also sees that $\delta$ can indeed be much larger than what's needed for resolving $\mathcal{O}(1/\omega)$ timescale oscillations. In addition, $\omega$ only needs to be larger than $H(Q(0),P(0))$ in magnitude.

Agreements with other aspects of the error bound are illustrated in tables \ref{tab_errorDependOmega} and \ref{tab_errorDependDelta}. The fact that the error is high-order in $\delta$ but only 1st-order in $\omega$ is another important property, as it leads to simultaneous accuracy and efficiency when $l$ is large ($l=4$ here).

\begin{table}[h]
  \begin{tabular}{r | c | c | c | c}
	  $\omega = $ & 20 & 40 & 80 & 160 \\
	  \hline
      max amplitude error $\approx$ & $6.2\times 10^{-8}$ & $1.2\times 10^{-7}$ & $2.5\times 10^{-7}$ & $5\times 10^{-7}$ \\
      \hline
      max phase error $\approx$ & $5.6\times 10^{-8}$ & $1.1\times 10^{-7}$ & $2.2\times 10^{-7}$ & $4.5\times 10^{-7}$
  \end{tabular}
  \caption{\footnotesize $\phi_4^\delta$ error is proportional to $\omega$. $T=100$ and $\delta=0.001$ fixed, $Q(0)=-3,P(0)=0$.}
  \label{tab_errorDependOmega}
\end{table}

\begin{table}[h]
  \begin{tabular}{r | c | c | c | c | c}
      $\delta = $ & $10^{-1}$ & $10^{-1.5}$ & $10^{-2}$ & $10^{-2.5}$ & $10^{-3}$ \\
      \hline
      max amplitude error $\approx$ & $0.76$ & $5.8\times 10^{-2}$ & $6.1\times 10^{-4}$ & $6.2\times 10^{-6}$ & $6.2\times 10^{-8}$ \\
      \hline
      max phase error $\approx$ & $0.76$ & $5.2\times 10^{-2}$ & $5.6\times 10^{-4}$ & $5.6\times 10^{-6}$ & $5.6\times 10^{-8}$
  \end{tabular}
  \caption{\footnotesize $\phi_4^\delta$ error is proportional to $\delta^4$. $T=100$ and $\omega=20$ fixed, $Q(0)=-3,P(0)=0$. The $\delta=0.1$ column anomaly is because $T\delta^l\omega$ is too large to be in the asymptotic regime of the error bound.}
  \label{tab_errorDependDelta}
\end{table}

\begin{figure}[h!]
\centering
\hspace{-10pt}
\subfigure[2nd-order version of the proposed method $\phi_2^\delta$]{
\includegraphics[width=0.5\textwidth]{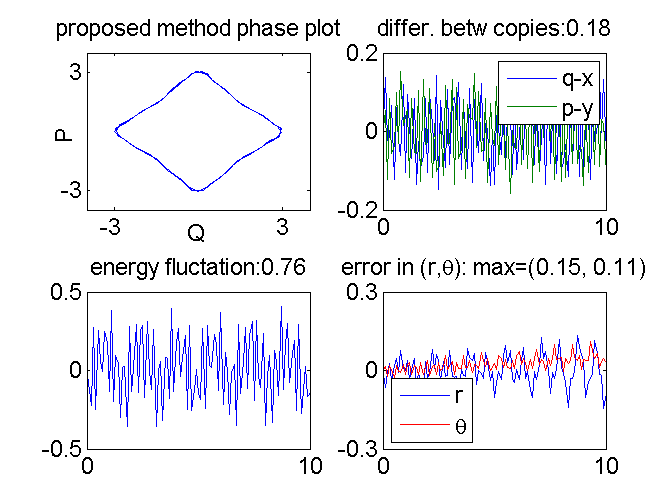}
\label{fig_1DOF_comparison_new2nd}
}
\hspace{-15pt}
\subfigure[recommended method in \cite{pihajoki2015explicit}]{
\includegraphics[width=0.5\textwidth]{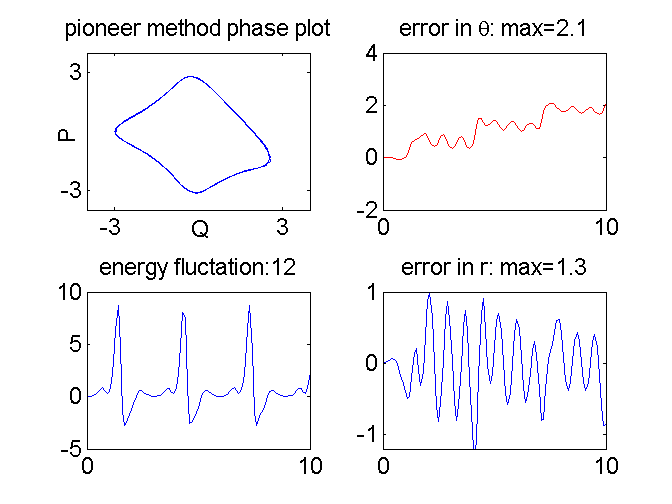}
\label{fig_1DOF_comparison_RK4}
}
\caption{\footnotesize Comparison between a low order version of the proposed integrator and the pioneer nonsympletic method. Timestep $\delta=0.1$ for both methods, $Q(0)=-3,P(0)=0$, $\omega=20$, and $T=10$.}
\label{fig_1DOF_comparison2}
\end{figure}

Figure \ref{fig_1DOF_comparison2} compares the pioneer method recommended in \cite{pihajoki2015explicit} ($Q\tilde{P}\tilde{Q}P$ with optimized $P_1$ mixing and projection; 2nd-order) with the proposed integrator (eq.\ref{eqn_2ndOrderIntegrator}; here we chose a low-order version so that the focus is on the benefit of symplecticity). Important to recall is the pioneer method is not symplectic because additional mixing and projection steps were introduced for improved accuracy (without such steps the accuracy time span will be even shorter; results not shown), but such steps breaks symplecticity. It is thus not surprising the pioneer method remains accurate only for a short time. Note however that the pioneer method is symmetric and exhibits no `secular' energy deviation.

\subsection{A second numerical example}
Consider the Schwarzschild geodesics problem simulated in \cite{pihajoki2015explicit} (with typos in the Hamiltonian, initial values, and precession rate estimation corrected). The geodesic can be cast as the solution of the 3 degrees of freedom ($Q=[t,r,\phi]$ and $P=[p_t,p_r,p_\phi]$) Hamiltonian system governed by
\[
	H=\frac{1}{2}\left[ \left( 1-\frac{2}{r} \right)^{-1} p_t^2-\left( 1-\frac{2}{r} \right)p_r^2-\frac{p_\phi^2}{r^2} \right].
\]

\begin{figure}[h!]
\centering
\hspace{-10pt}
\subfigure[Geodesic solution]{
\includegraphics[width=0.5\textwidth]{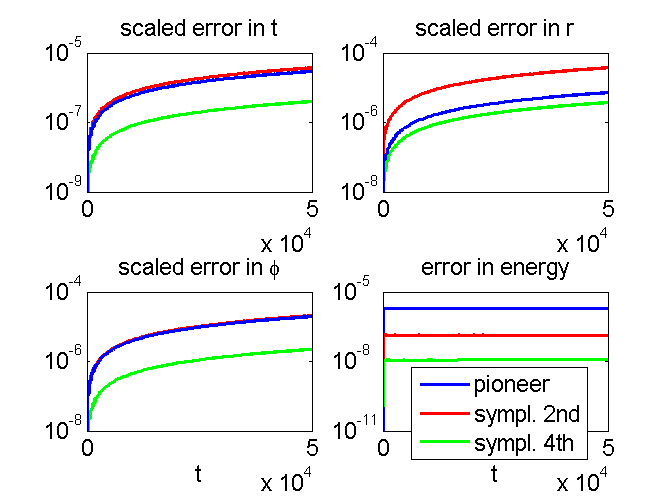}
\label{fig_geodesic_conservative}
}
\hspace{-15pt}
\subfigure[Dissipated solution ($\gamma=10^{-4}$)]{
\includegraphics[width=0.5\textwidth]{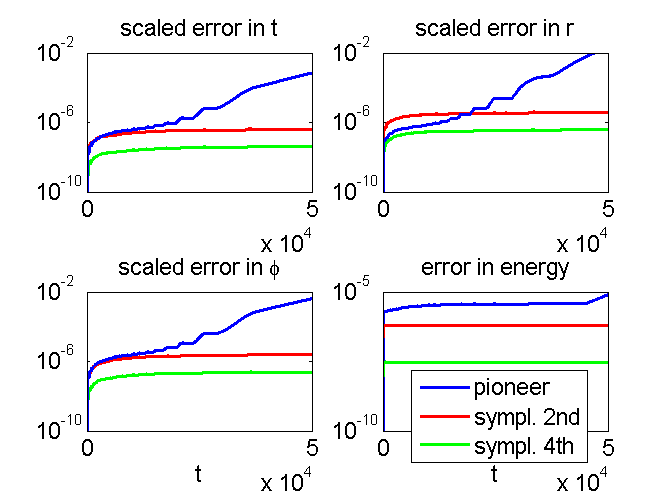}
\label{fig_geodesic_dissipative}
}
\caption{\footnotesize Comparison between the proposed integrators and the pioneer nonsympletic method. Timestep $\delta=0.2$ for both methods, $\omega=2$, $T=50000$, $M=10$, $m=1$, initial condition $Q(0)=[0,20,0], P(0)\approx [0.982,0,-4.472]$ corresponds to initial semi-major axis $a(0)=20$, eccentricity $e(0)=0$, $p_r(0)=0$, $p_\phi(0)=\sqrt{r(0)}$, and $p_t(0)$ solves $H(0)=m^2/2$. Plotted were maxima of relative errors up to given times, respectively scaled by dividing over Keplerian period $2\pi\sqrt{a(0)^3/M}$ for $t$, over $a(0)(1+e(0))$ for $r$, over $2\pi$ for $\phi$, and unscaled for $H$.}
\label{fig_geodesic}
\end{figure}

Figure \ref{fig_geodesic_conservative} estimates accuracies of Schwarzschild geodesic computed by 2nd- and 4th-order proposed integrators $\phi_2^\delta$ (eq.\ref{eqn_2ndOrderIntegrator}) and $\phi_4^\delta$ (eq.\ref{eqn_highOrderIntegrator}), as well as that by the pioneer nonsymplectic method recommended in \cite{pihajoki2015explicit} ($Q\tilde{P}\tilde{Q}P$ with $P_1$ mixing and projection, which was the optimized choice for this problem). Note accuracy can only be estimated because no exact solution is available for quantifying the numerical error, and we used the adaptive MATLAB ode45 with relative and absolute error tolerances both set to $10^{-20}$ to generate a benchmark for error estimation; although the long time fidelity of this benchmark cannot be guaranteed, throughout the simulation it only produced $< 7\times 10^{-15}$ deviation from the conserved energy value. One sees $\phi_4^\delta$ is the most accurate of the three in all aspects, while $\phi_2^\delta$ is less accurate on $r$ than the optimized pioneer method despite of its better energy preservation. Worth mentioning is, for this example, the optimized pioneer method allows larger step sizes than $\phi_l^\delta$ (results not shown).

We then evaluate how each method captures the effect of additional non-conservative forces, by adding a simplest-possible dissipation, i.e. consider $\dot{q}=\partial H/\partial p, \dot{p}=\partial H/\partial q-\gamma p$. To add external forces in the proposed and pioneer integrators, we replace the $p$ update in $\phi_{H_A}^\delta$ and the $y$ update in $\phi_{H_B}^\delta$ (eqn.\ref{eqn_flowMaps}) by
\[
	p \mapsto p+\delta (-\partial_q H(q,y) - \gamma y), \qquad
	y \mapsto y+\delta (-\partial_x H(x,p) - \gamma p).
\]
Integration accuracies are again estimated by comparing to an adaptively integrated fine benchmark (Figure \ref{fig_geodesic_dissipative}). With dissipation, long time errors of the symplectic versions are much smaller than the pioneer method. The intuition is, as the force breaks time reversibility, the symmetry of the pioneer method is no longer advantageous; symplectic integrators, however, are known to well approximate changes in phase-space volume \cite{BoOw:09,ObTa13}. In fact, in a longer simulation ($T=10^5$), the pioneer method became unstable while other methods remained bounded and relatively accurate (results not shown).

\section{Non-integrable system: a numerical demonstration on the weak turbulent nonlinear Schr\"{o}dinger equation}
\label{sec_nonIntegrableCase}
It is known that nonlinear Schr\"{o}dinger equation is non-integrable in $\geq 2$ spatial dimensions and exhibits weak turbulence \cite{dyachenko1992optical}. 
Conditions under which a finite-dimensional Hamiltonian system of Fourier coefficients can approximate the nonlinear Schr\"{o}dinger equation on 2-torus were rigorously demonstrated in \cite{colliander2010transfer}, while extensive usages of similar models preceded this rigorous justification \cite{cai1999spectral, cai2001dispersive,majda1997one,milewski2002resonant,deville2007nonequilibrium}. This system has a nonseparable Hamiltonian that can be written as
\[
	H(q,p)=\frac{1}{4}\sum_{i=1}^N \left( q_i^2+p_i^2 \right)^2-\sum_{i=2}^N \left( p_{i-1}^2 p_i^2 + q_{i-1}^2 q_i^2 - q_{i-1}^2 p_i^2 - p_{i-1}^2 q_i^2 + 4 p_{i-1}p_i q_{i-1}q_i \right).
\]
Note an explicit symplectic integrator was proposed in \cite{blanes2002symplectic} for polynomial Hamiltonians, which suit this system. That method requires more computations per timestep because it's based on splitting the Hamiltonian into mononials, but it is symplectic in the original phase space, which is certainly an advantage. Our purpose is not to compare that method with the more general method proposed here, but only to numerically access whether our method still has good long time performances for a non-integrable system.

To do so, we focus on first integrals and the statistical behavior of long-time dynamics. If one denotes by $I_i=q_i^2+p_i^2$ the mass of each mode, the total mass $I:=\sum I_i$ can be shown as a second first integral of the system in addition to energy conservation. Also, although mathematically difficult to prove, it is widely believed the system, due its turbulent nature, is ergodic on first integral foliations. We thus assume an ergodic measure of $\delta(H(q,p)-E)\delta(\sum I_i-I) \, dq \, dp$. It is easy to show this constrained Liouville measure is at least an invariant measure. Under the ergodicity assumption, long time averages of phase space observables converge to their spatial averages with respect to the ergodic measure, and mode relabeling symmetry of $H$ leads to, if $N=2$, that
\begin{equation}
	\lim_{T\rightarrow\infty} \langle I_1 \rangle (T) = \lim_{T\rightarrow\infty} \langle I_2 \rangle (T) = \frac{1}{2}I, \text{ where } \langle I_i \rangle (T)=\frac{1}{T} \int_0^T I_i(t) dt.
	\label{eq_ergodicConvergence}
\end{equation}

\begin{figure}
\centering
\includegraphics[width=\textwidth]{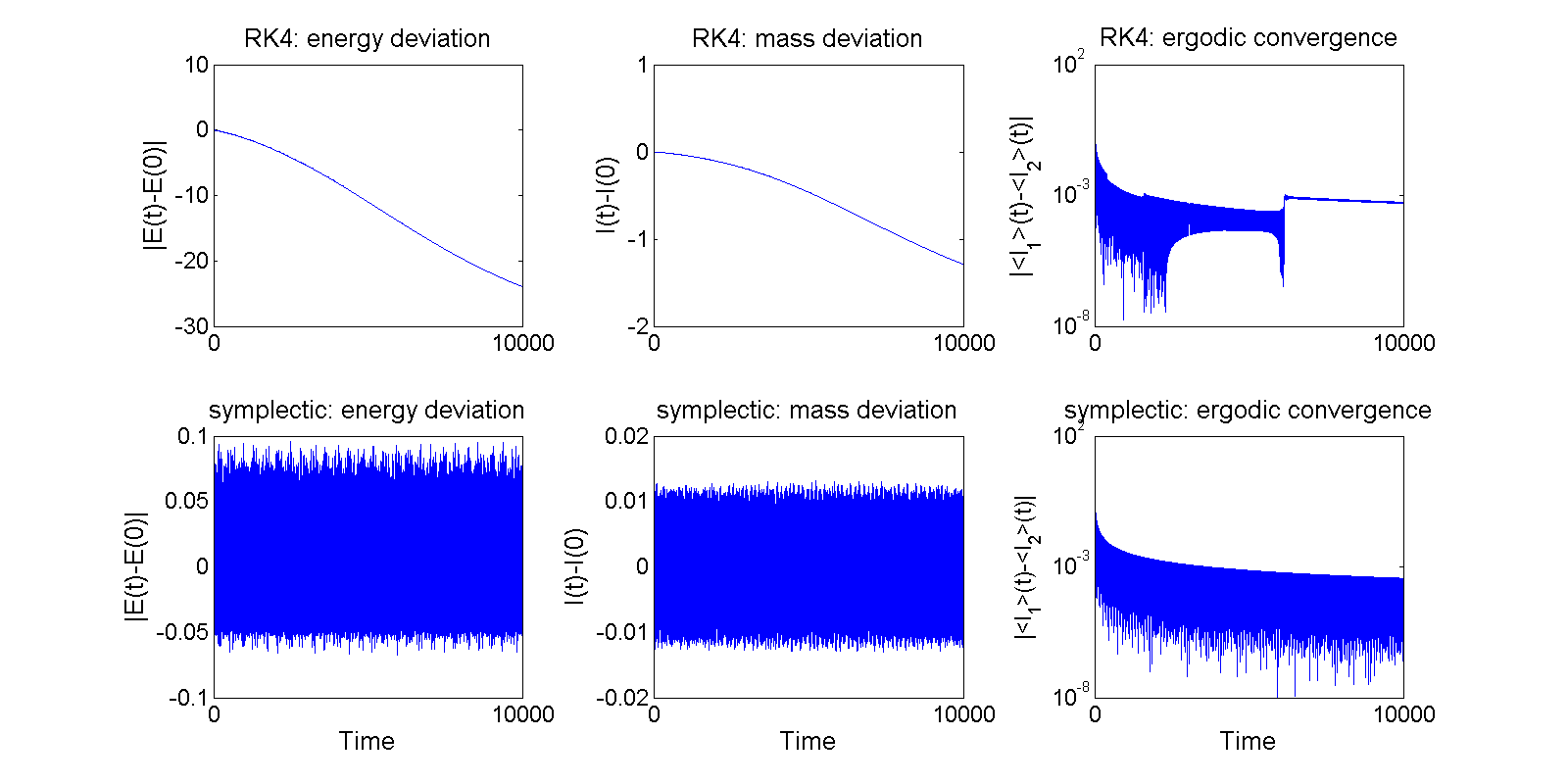}
\caption{\footnotesize Long time simulations of NLS with $N=2$ by 4th-order Runge-Kutta and a proposed integrator $\phi_4^\delta$. $[q_1,p_1,q_2,p_2](0)=[3,1,0.01,0]$, $\delta=0.01$, $\omega=100$. Long time convergence of the trajectory towards the ergodic limit is quantified in the right column based on $\langle I_1 \rangle (t)-\langle I_2 \rangle (t)$ (see eq. \ref{eq_ergodicConvergence}); this observable was chosen to reduce the interference with the numerical loss of total mass in RK4. Computation done using win32 MATLAB R2010a on x64 Windows 7 i7-4600U CPU.}
\label{fig_NLS_longTime}
\end{figure}

Figure \ref{fig_NLS_longTime} shows that the proposed method better captures total energy and mass conservations, as well as the assumed convergence toward ergodicity \eqref{eq_ergodicConvergence}. The standard non-symplectic Runge-Kutta seems to be less accurate in capturing the assumed equilibration, which normally happens at a more consistent rate. In this sense, even though the error analysis in section \ref{sec_integrableCase} doesn't apply to non-integrable systems, the proposed integrator still exhibits better long-time performance than its non-symplectic counterpart, at least in this example.

\begin{figure}
\centering
\includegraphics[width=\textwidth]{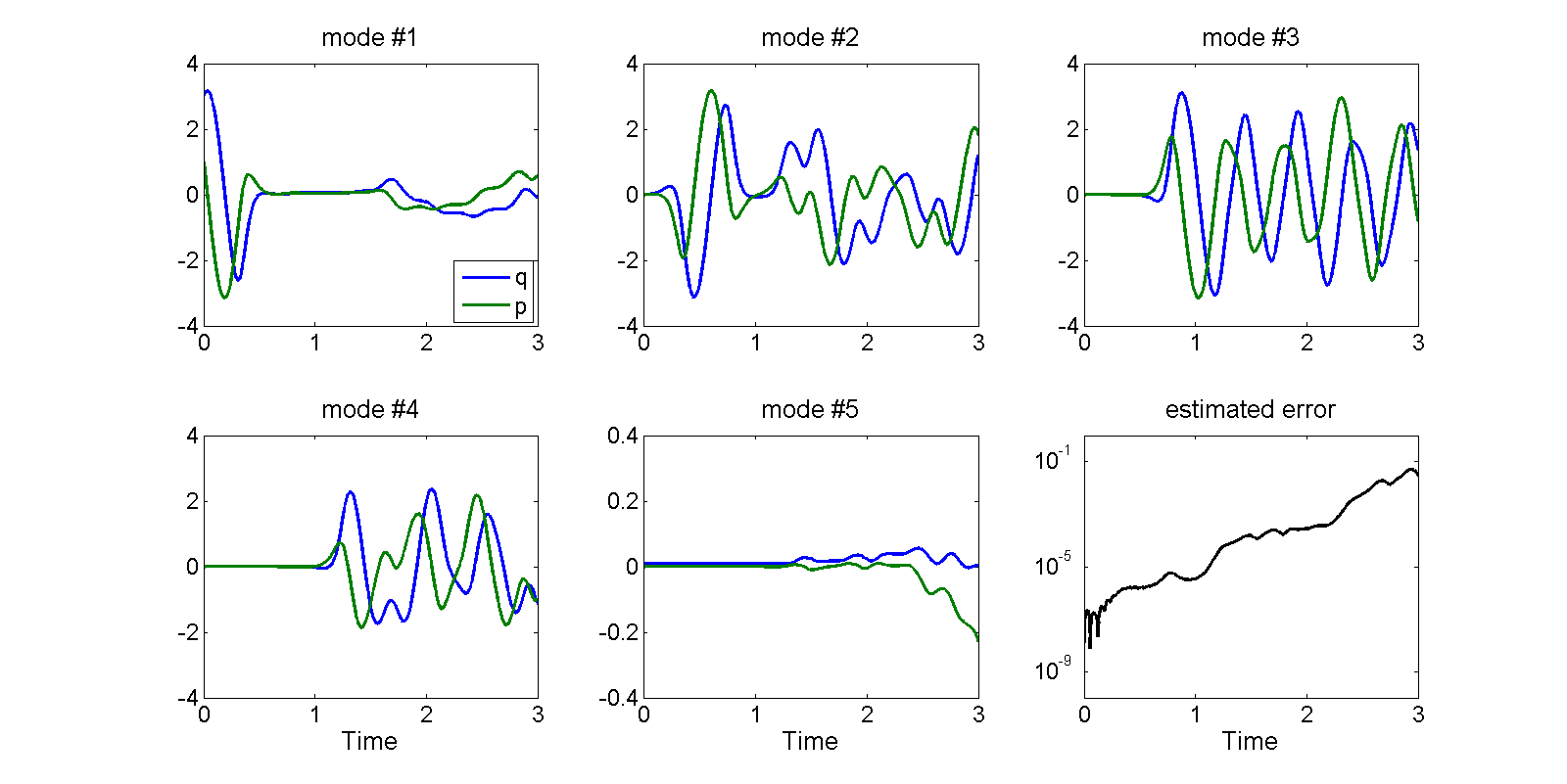}
\caption{\footnotesize Short time simulation of NLS with $N=5$ by $\phi_4$. $[q_1,p_1](0)=[3,1]$, $[q_i,p_i](0)=[0.01,0]$ for $i\neq 1$, $\delta=0.001$ for a closer examination of details, $\omega=100$. Error is estimated by comparing to fine RK4 simulation with $\delta=10^{-4}$.}
\label{fig_NLS_shortTime}
\end{figure}

For completeness, figure \ref{fig_NLS_shortTime} also illustrates an $\mathcal{O}(1)$ time simulation of the proposed method for a larger $N$, where weak turbulent cascade is clearly observed.

\section{On integrability of the extended system}
\label{sec_integrabilityExtendedSys}
As discussed above, symplectic integrators have desirable long-time performances for integrable systems (see \cite{calvo1995accurate, quispel1998volume,Hairer06} for non-stiff problems, and section \ref{sec_integrableCase} for our specific stiff problem). However, if the system has a positive Lyapunov exponent, symplectic integrator can be as bad as a generic integrator, because its error can grow exponentially with time as local truncation error propagates along.

\begin{figure}
\centering
\hspace{-30pt}
\subfigure[no restraint.]{
\includegraphics[width=0.4\textwidth]{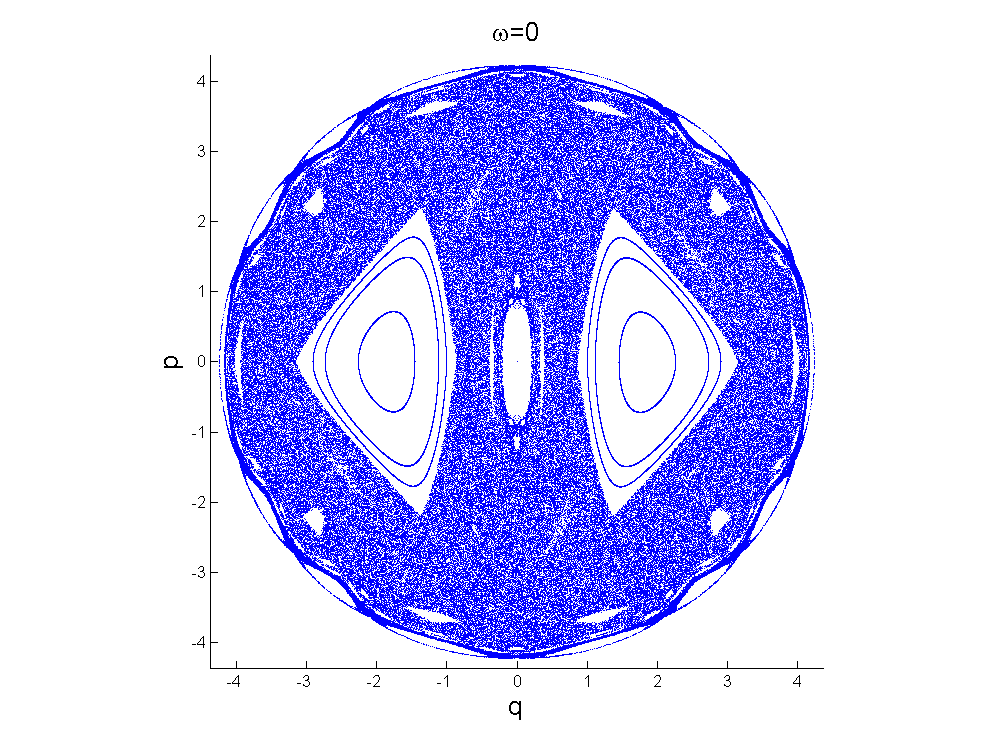}
\label{fig_extendedIntegrability_w0}
}
\hspace{-45pt}
\subfigure[small restraint.]{
\includegraphics[width=0.4\textwidth]{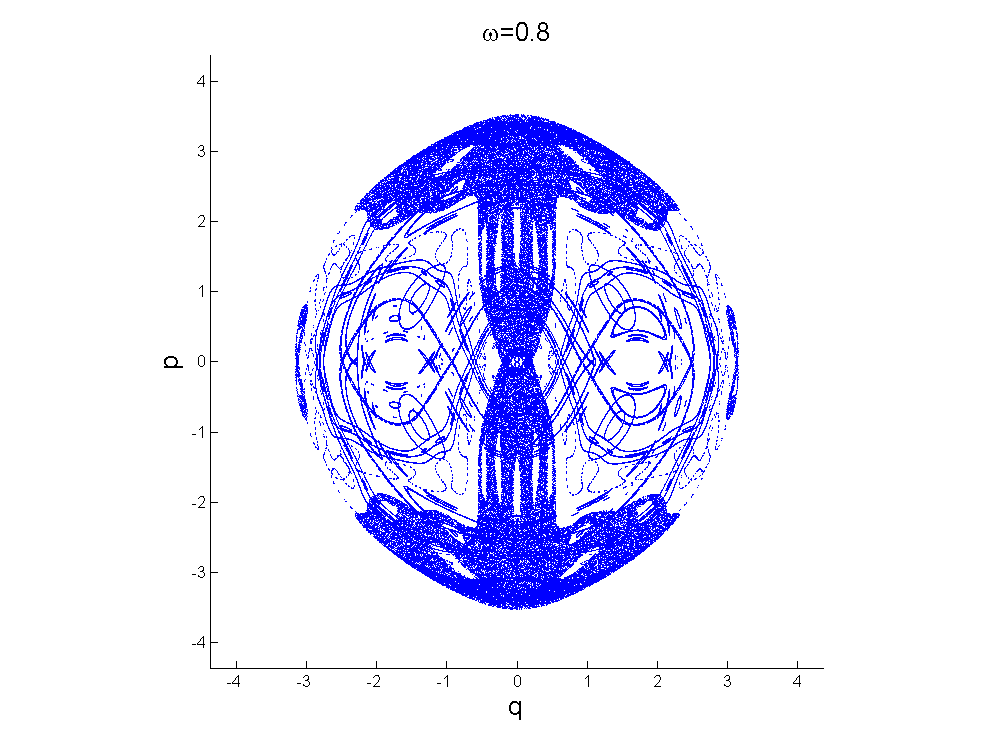}
\label{fig_extendedIntegrability_w0p8}
}
\hspace{-45pt}
\subfigure[strong restraint.]{
\includegraphics[width=0.4\textwidth]{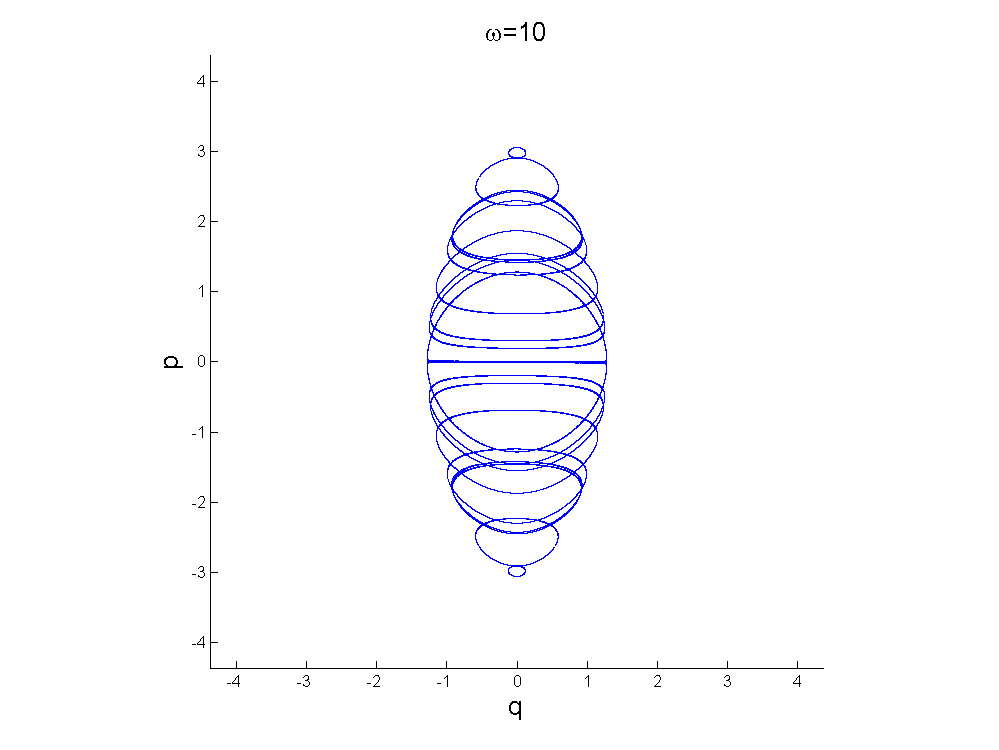}
\label{fig_extendedIntegrability_w10}
}
\hspace{-30pt}
\caption{\footnotesize Poincar\'e section of $\bar{H}$ with $H(Q,P)=(Q^2+1)(P^2+1)/2$ at surface $x=0$ on constant energy shell $\bar{H}=10$. Initial conditions are uniformly sampled in $q,p$ plane. Note curves do intersect because each $q,p$ pair corresponds to two $y$ values.}
\label{fig_extendedIntegrability}
\end{figure}

Unfortunately, the system governed by $H(q,y)+H(x,p)$ in extended phase space may not be integrable even if $H(Q,P)$ is integrable in the original phase space. For instance, consider the 1 degree of freedom problem 
studied in section \ref{sec_1DOF}. The original system is integrable because $H$ is a first integral. Figure \ref{fig_extendedIntegrability} illustrates one Poincar\'e section of the $\omega=0$ extended system, where large chaotic seas demonstrate non-integrability. This explains why a symplectic integrator for $\bar{H}$ with $\omega=0$, as considered in \cite{pihajoki2015explicit}, does not have satisfactory performances beyond $\mathcal{O}(1)$ time.

However, $H_C$ is integrable. Thus, as $\omega$ increases, a larger proportion of the phase space for $\bar{H}=H_A+H_B+\omega H_C$ corresponds to regular behaviors (see \cite{kolmogorov1954conservation,arnol1963proof,moser1962invariant,poschel1982integrability,eliasson1996absolutely,
giorgilli1997kolmogorov}). Indeed, Poincar\'e sections in figure \ref{fig_extendedIntegrability_w0p8},\ref{fig_extendedIntegrability_w10} show smaller chaotic seas when $\omega=0.8$, and no evidence of chaotic sea but only invariant tori when $\omega=10$. This suggests that the integrability assumption in section \ref{sec_integrableCase} is reasonable as long as the original system is integrable and $\omega$ is larger than a threshold $\omega_0$, with exception initial conditions whose measure vanishes. 

\section{Appendix: refined estimates of $\alpha$ and $\beta$}
\label{sec_appRefinedEstimate}
This section shows that $\alpha,\beta$ are not just $\mathcal{O}(1/\sqrt{\omega})$ but in fact $\mathcal{O}(1/\omega)$ at least for a long time. As this is a self-contained section, we set up the problem in the original time, denote by $t$ the corresponding time variable (as opposed to $T$ in the main text) and by dot the $t$ derivative; the Hamiltonian is 
\begin{equation}
	\hat{H}=H\left(\bar{Q}+\frac{\alpha}{2},\bar{P}-\frac{\beta}{2}\right) + H\left(\bar{Q}-\frac{\alpha}{2},\bar{P}+\frac{\beta}{2}\right) + \frac{1}{2} \omega \|\alpha\|^2 + \frac{1}{2} \omega \|\beta\|^2 + \mathcal{R}.
	\label{eq_Hhat}
\end{equation}
We will introduce a near identify transformation to express the governing Hamilton's equations in a nonstandard 2nd-order normal form, which we will utilize to refine the estimates of $\alpha$, $\beta$. To begin, note $\omega J$ is the coefficient matrix associated with the linear dynamical system generated by $H_0=\frac{1}{2} \omega \alpha^2 + \frac{1}{2} \omega \beta^2$, where $J=\begin{bmatrix} 0 & I \\ -I & 0 \end{bmatrix}$, and $\exp(t \omega J)=\mathcal{O}(1)$ for all real $t$. We use $x\in \mathcal{R}^{2d}$ to represent $[\alpha, \beta]$.

\begin{Lemma}
	Given a $2d$-dimensional real symmetric-matrix valued function $S(s)$,
	\[
		\Omega(s):=\frac{1}{2\pi}\int_0^{2\pi} e^{-J \tau} J S(s) e^{J \tau} \, d\tau
	\]
	is a real skew-symmetric matrix valued function. Moreover, its associated fundamental matrix $\Phi(s)$, defined as the solution of
	\[
		\Phi'(s)=\Omega(s) \Phi(s),	\quad \Phi(0)=I,
	\]
	satisfies $\Phi(s)=\mathcal{O}(1)$ and $\Phi(s)^{-1}=\mathcal{O}(1)$ for all $s$.
	\label{lemma_estimateResonantLinearPart}
\end{Lemma}
\begin{proof}
	Assume $S$ in block form is $S=\begin{bmatrix} A & B \\ B^T & D \end{bmatrix}$ where $A^T=A, D^T=D$. Then
	\begin{align*}
		\Omega &= \frac{1}{2\pi}\int_0^{2\pi} \begin{bmatrix} I \cos t & -I \sin t \\ I \sin t & I \cos t \end{bmatrix}
			\begin{bmatrix} B^T & D \\ -A & -B \end{bmatrix} 
			\begin{bmatrix} I \cos t & I \sin t \\ -I \sin t & I \cos t \end{bmatrix} \, dt \\
			&= \frac{1}{2} \begin{bmatrix} B^T-B & A+D \\ -(A+D) & B-B^T \end{bmatrix},
	\end{align*}
	and it is real skew-symmetric.
	
	To bound $\Phi(s)$, note $(\Phi^T \Phi)'=(\Omega \Phi)^T \Phi+\Phi^T (\Omega\Phi)=0$ and thus $\|\Phi(s)\| =\|\Phi(0)\|=1$.
	
	$\Phi(s)^{-1}$ can be similarly bounded since it is easy to verify that $d\Phi^{-1}/ds=-\Phi^{-1}\Omega$.
\end{proof}

\begin{Theorem}
	Consider in $\mathcal{R}^{2d}$ an ODE
	\begin{align*}
		\dot{x} &= \omega J x + F_0(t) + F_1(t) x + \sum_{i=1}^{2d} x^T F_{2i}(t) x \textbf{e}_i + \mathcal{O}(x^3),
	\end{align*}
	where solutions are assumed to exist and $x$ satisfies $x=\mathcal{O}(1/\sqrt{\omega})$, $x(0)=\textbf{0}$, $\textbf{e}_i$'s are standard basis of $\mathcal{R}^{2d}$, $F_0, F_1, F_{2i}$ have bounded derivatives, and $F_1(t)=J S(t)$ for some symmetric-matrix-valued $S(t)$. Then
	\[
		x(t) = \mathcal{O}(1/\omega), \quad \text{till at least } t=\mathcal{O}(\sqrt{\omega}).
	\]
	\label{thm_normalForm}
\end{Theorem}
\begin{proof}
	Let $y(t)=\exp(-\omega J t) x(t)$, then
	\[
		\dot{y} = e^{-\omega J t} F_0(t) + e^{-\omega J t} F_1(t) e^{\omega J t} y + e^{-\omega J t} \sum_{i=1}^{2d} y^T e^{-\omega J t} F_{2i}(t) e^{\omega J t} y \textbf{e}_i + \mathcal{O}(\omega^{-3/2}).
	\]
	Let $\epsilon=1/\omega$, introduce a dummy slow variable $s=t$ and a fast variable $\tau=\omega t$ (which corresponds to the angle associated with the rotation in $\alpha, \beta$), and use prime to denote $d/d\tau$, then the governing equation rewrites as a slow/fast system
	\begin{align}
		y' &= \epsilon e^{-J \tau} F_0(s) + \epsilon e^{-J \tau} F_1(s) e^{J \tau} y + \epsilon e^{-\omega J t} \sum_{i=1}^{2d} y^T e^{-J \tau} F_{2i}(s) e^{J \tau} y \textbf{e}_i + \mathcal{O}(\epsilon^{5/2}) \label{eq_yNormalForm} \\
		s' &= \epsilon \nonumber \\
		\tau' &= 1 \nonumber 
	\end{align}
	We look for a near-identity transformation in the form of
	\[
		y=z+\epsilon u_1(z,s,\tau)+\epsilon^2 u_2(z,s,\tau)+\mathcal{O}(\epsilon^{5/2}),
	\]
	where $u_1,u_2$ are $2\pi$-periodic in $\tau$, such that
	\begin{align*}
		z' &= \epsilon f_1(z,s)+\epsilon^2 f_2(z,s) + \epsilon S + \mathcal{O}(\epsilon^{5/2}) \\
		S &:= e^{-\omega J t} \sum_{i=1}^{2d} z^T e^{-J \tau} F_{2i}(s) e^{J \tau} z \textbf{e}_i
	\end{align*}
	for some $f_1, f_2$ independent of the fast variable $\tau$.
	
	Under this transformation, the right hand side of \eqref{eq_yNormalForm} becomes
	\begin{align*}
		&~ \epsilon e^{-J \tau} F_0(s) + \epsilon e^{-J \tau} F_1(s) e^{J \tau} z + \epsilon^2 e^{-J \tau} F_1(s) e^{J \tau} u_1 + \epsilon e^{-\omega J t} \sum_{i=1}^{2d} z^T e^{-J \tau} F_{2i}(s) e^{J \tau} z \textbf{e}_i \\
		& \qquad\qquad + \epsilon^2 e^{-\omega J t} \sum_{i=1}^{2d} (z^T e^{-J \tau} F_{2i}(s) e^{J \tau} u_1 + u_1^T e^{-J \tau} F_{2i}(s) e^{J \tau} z) \textbf{e}_i + \mathcal{O}(\epsilon^{5/2}) \\
		&= \epsilon e^{-J \tau} F_0(s) + \epsilon e^{-J \tau} F_1(s) e^{J \tau} z + \epsilon^2 e^{-J \tau} F_1(s) e^{J \tau} u_1 + \epsilon S + \mathcal{O}(\epsilon^{5/2}),
	\end{align*}
	where the last equality is due to $z=y+\mathcal{O}(\epsilon)=\mathcal{O}(1)x + \mathcal{O}(\epsilon)=\mathcal{O}(\epsilon^{1/2})$.
	
	The left hand side of \eqref{eq_yNormalForm}, on the other hand, becomes
	\begin{align*}
		&~ z' + \epsilon \frac{\partial u_1}{\partial z} z' + \epsilon \frac{\partial u_1}{\partial s} s' + \epsilon \frac{\partial u_1}{\partial \tau} \tau' + \epsilon^2 \frac{\partial u_2}{\partial z} z' + \epsilon^2 \frac{\partial u_2}{\partial s} s' + \epsilon^2 \frac{\partial u_2}{\partial \tau} \tau' + \mathcal{O}(\epsilon^{5/2}) \\
		&= \epsilon f_1 + \epsilon^2 f_2 + \epsilon S + \epsilon \frac{\partial u_1}{\partial z} \epsilon f_1 + \epsilon \frac{\partial u_1}{\partial z} \epsilon S + \epsilon^2 \frac{\partial u_1}{\partial s} + \epsilon \frac{\partial u_1}{\partial \tau} + \epsilon^2 \frac{\partial u_2}{\partial \tau} + \mathcal{O}(\epsilon^{5/2}) \\
		&= \epsilon f_1 + \epsilon^2 f_2 + \epsilon S + \epsilon^2 \frac{\partial u_1}{\partial z} f_1 + \epsilon^2 \frac{\partial u_1}{\partial s} + \epsilon \frac{\partial u_1}{\partial \tau} + \epsilon^2 \frac{\partial u_2}{\partial \tau} + \mathcal{O}(\epsilon^{5/2}),
	\end{align*}
	where the last equality is due to that $S=\mathcal{O}(z^2)=\mathcal{O}(\epsilon^{-1})$.
	
	Matching $\mathcal{O}(\epsilon)$ and $\mathcal{O}(\epsilon^2)$ terms respectively, we obtain the following requirements on $u_1,u_2,f_1,f_2$:
	\begin{align}
		& f_1+\frac{\partial u_1}{\partial \tau} = e^{-J\tau} F_0(s) + e^{-J\tau} F_1(s) e^{J\tau} z , \label{eq_normalFormCondition1} \\
		& f_2+\frac{\partial u_1}{\partial z} f_1 + \frac{\partial u_1}{\partial s} + \frac{\partial u_2}{\partial \tau} = e^{-J\tau} F_1(s) e^{J\tau} u_1 . \label{eq_normalFormCondition2}
	\end{align}
	In order for a solution $u_1$ periodic in $\tau$ to exist, $f_1$ has to satisfy a solvability condition
	\[
		2\pi f_1(z,s) = \int_0^{2\pi} \left( e^{-J\tau} F_0(s) + e^{-J\tau} F_1(s) e^{J\tau} z \right) \, d\tau,
	\]
	which is obtained after integrating both sides of \eqref{eq_normalFormCondition1} over $\tau$. The first term of the integrand vanishes after integration, and lemma \ref{lemma_estimateResonantLinearPart} then leads to
	\[
		f_1(z,s)=\Omega(s) z
	\]
	for some real skew-symmetric $\Omega$. Integration then gives
	\[
		u_1=e^{-J\tau} J F_0(s) + g(\tau,s) z
	\]
	for some $g(\tau,s)$ periodic in $\tau$.
	
	The next order equation \eqref{eq_normalFormCondition2} leads to a solvability condition
	\[
		2\pi f_2(z,s) = \int_0^{2\pi} \left( e^{-J\tau} F_1(s) e^{J\tau} u_1 - \frac{\partial u_1}{\partial s} - \frac{\partial u_1}{\partial z} f_1 \right) \, d\tau,
	\]
	and thus $f_2=-\langle \frac{\partial g}{\partial s} \rangle z - \langle g \rangle \Omega z$, where $\langle \cdot \rangle$ indicates time average with respect to $\tau$.
	
	Consequently, $z$ satisfies
	\[
		z'=\epsilon \Omega (s) z+\epsilon e^{-\omega J t} \sum_{i=1}^{2d} z^T e^{-J \tau} F_{2i}(s) e^{J \tau} z \textbf{e}_i - \epsilon^2 \langle \frac{\partial g}{\partial s} \rangle z - \epsilon^2 \langle g \rangle \Omega z + \mathcal{O}(\epsilon^{5/2}).
	\]
	Rescale back to the original time so that the right hand side gets divided by $\epsilon$. Let $Z(t)=\Phi(t)^{-1} z(t)$, where $\Phi$ is the fundamental matrix associated with $\Omega$. Then 
	\[
		\dot{Z}= \Phi(t)^{-1} e^{-\omega J t} \sum_{i=1}^{2d} (\Phi(t) Z(t))^T e^{-J \tau} F_{2i}(s) e^{J \tau} (\Phi(t) Z(t)) \textbf{e}_i - \epsilon \Phi(t)^{-1} (\langle \frac{\partial g}{\partial s} \rangle + \langle g \rangle \Omega ) \Phi(t) Z(t) + \mathcal{O}(\epsilon^{3/2})
	\]
	Note $\Phi$ and $\Phi^{-1}$ are bounded according to lemma \ref{lemma_estimateResonantLinearPart}.
	Rewrite this in integral form, and bounding one of the two $Z$'s in the first (quadratic) term by $Z=\mathcal{O}(z)=\mathcal{O}(\epsilon^{1/2})$, we have
	\[
		\|Z(t)\| \leq \int_0^t \left( \epsilon^{1/2} C_1 \|Z(\hat{t})\| + \epsilon C_2 \|Z(\hat{t})\| + \epsilon^{3/2} C_3 \right) \, d\hat{t}
	\]
	for some $C_1,C_2,C_3>0$ independent of $\epsilon$.
	Gronwall's lemma leads to
	\[
		\|Z(t)\| \leq \epsilon^{3/2} t C_3 \exp((\epsilon^{1/2} C_1+\epsilon C_2)t),
	\]
	which means $Z(t)=\mathcal{O}(\epsilon)$ till at least $t=\mathcal{O}(\epsilon^{-1/2})$. Since $Z$, $z$, $y$, $x$ are at the same order due to boundedness of $\Phi$ and $\exp(\omega J t)$, $x(t)=\mathcal{O}(\epsilon)=\mathcal{O}(1/\omega)$ till at least $t=\mathcal{O}(\epsilon^{-1/2})=\mathcal{O}(\sqrt{\omega})$.
\end{proof}

\begin{Theorem}
	Consider the dynamics generated by $\hat{H}$ (eqn. \ref{eq_Hhat}), where $\mathcal{R}$ is given by a sum of nested Poisson brackets (see step \ref{item_step2} in section \ref{sec_integrableCase}). If $H(\cdot,\cdot)$ is analytic, $\omega$ is large enough, $\alpha(0)=\beta(0)=\textbf{0}$, $\alpha(t),\beta(t)=\mathcal{O}(\omega^{-1/2})$ till at least $t=\mathcal{O}(\sqrt{\omega})$, and $\bar{Q}(t),\bar{P}(t)$ remain bounded independent of $\omega$ till at least the same time, then $\alpha(t),\beta(t)=\mathcal{O}(\omega^{-1})$ till at least the same time.
\end{Theorem}

\begin{proof}
Let $x=[\alpha,\beta]$, then
\[
	\dot{x}=J \frac{\partial H}{\partial x}.
\]
Since $H$ and $\alpha^2+\beta^2$ are analytic, $\mathcal{R}$ is also analytic because it is a sum of their nested canonical Poisson brackets. Therefore, $\hat{H}$ can be locally written as
\[
	\hat{H}=H_0(\bar{Q},\bar{P})+H_1(\bar{Q},\bar{P}) x+\frac{1}{2} x' H_2(\bar{Q},\bar{P}) x+\sum_{n=3}^{\infty} H_n(\bar{Q},\bar{P})[x^{\otimes}],
\]
where $H_1(\bar{Q},\bar{P})$ is a vector, $H_2(\bar{Q},\bar{P})$ is a symmetric matrix, and $H_n(\bar{Q},\bar{P})[x^{\otimes n}]$ stands for a homogeneous polynomial of degree $n$ in $x$, with coefficients being analytic functions of $\bar{Q},\bar{P}$.

Therefore, the $\alpha,\beta$ dynamics can be locally written as
\[
	\dot{x}=\omega J x+F_0(\bar{Q},\bar{P})+F_1(\bar{Q},\bar{P})x+\sum_{i=1}^{2d} x^T F_{2i}(\bar{Q},\bar{P}) x \textbf{e}_i + \sum_{n=3}^\infty F_n(\bar{Q},\bar{P})(x^{\otimes n}),
\]
for some $F_1$, $F_{2i}$, and $F_n$. In particular, $F_1(\bar{Q},\bar{P})=J H_2(\bar{Q},\bar{P})$ for symmetric $H_2$.

Since $\bar{Q},\bar{P}$ are functions of $t$ and bounded (independent of $\omega$) till at least $t=\mathcal{O}(\omega^{1/2})$, $\dot{\bar{Q}},\dot{\bar{P}}$, which are given by Hamilton's equations as functions of $\bar{Q},\bar{P},\alpha,\beta$, are also bounded. Consequently, $F_n(\bar{Q},\bar{P})$ are implicitly functions of time (i.e. $F_n(t)$) with bounded 1st time derivatives.

Therefore, Theorem \ref{thm_normalForm} applies, and thus $x=\mathcal{O}(\omega^{-1})$ till at least $t=\mathcal{O}(\omega^{1/2})$.
\end{proof}

Note we worked at the level of equations but not Hamiltonians; i.e. we did not express \eqref{eq_Hhat} as a nearly integrable system by transforming $\alpha,\beta$ into action and angle. This is because we are interested in the special initial condition $\alpha(0)=\beta(0)=0$, which corresponds a singularity of the transformation, and it will break the analyticity of $\hat{H}$ and render classical approaches like KAM or Nekhoroshev's method difficult to apply.

\section{Acknowledgment}
This work was partially supported by NSF grant DMS-1521667. The author sincerely thanks Chengchun Zeng, Jonathan Goodman, and Rafael de la Llave for inspiring discussions, Zaher Hani for detailed explanations of nonlinear Schr\"{o}dinger equation, and Gongjie Li, Huan Yang and Pauli Pihajoki for rich knowledge in Schwarzschild geodesics.

\bibliographystyle{unsrt}
{\footnotesize
\bibliography{molei24}
}

\end{document}